\documentclass[12pt]{article}
\usepackage{url}

\usepackage{a4wide,epsfig,amsmath,amssymb,latexsym,color}
\usepackage{tikz}
\usepackage{bm}
\usepackage{comment}
\usepackage{xcolor}

\usepackage{tikz}
\usepackage{pict2e}
\definecolor{lightorange}{rgb}{0.91, 0.84, 0.42}
\definecolor{deepcarrotorange}{rgb}{0.91, 0.41, 0.17}
\definecolor{deepsaffron}{rgb}{1.0, 0.6, 0.2}
\definecolor{gold}{rgb}{0.91, 0.84, 0.17}

\definecolor{red}{rgb}{0.8,0,0}
\definecolor{darkorange}{rgb}{1,0.4,0}
\definecolor{lightorange}{rgb}{1,0.6, 0}
\definecolor{yellow}{rgb}{1,0.8, 0}

\def\RR{\mathbb{R}}

\def\mos{{\rm mos}\,}
\def\calT{{\cal T}}
\def\WF{{\rm WF}}

\newtheorem{theorem}{Theorem}
\newtheorem{lemma}{Lemma}

\newtheorem{definition}{Definition}
\newtheorem{corollary}{Corollary}
\newenvironment{proof}{\begin{trivlist}\item[]{\emph{Proof.}}}
               {\hfill$\Box$\end{trivlist}}

\begin{document}

\title{A characterization of supersmoothness \\
        of multivariate splines}
\author{Michael S. Floater\footnote{
   Department of Mathematics,
   University of Oslo, Moltke Moes vei 35, 0851 Oslo, Norway,
   {\it email: michaelf@math.uio.no}}
\and
Kaibo Hu\footnote{School of Mathematics, University of Minnesota, 
  206 Church St. SE, 
Minneapolis, MN 55455-0488, USA.
   {\it email: khu@umn.edu} }\thanks{KH was supported in part by the
European Research Council under the European Union's Seventh
Framework Programme (FP7/2007-2013) / ERC grant agreement 339643,
during his affiliation with the University of Oslo.}
}

\maketitle

\abstract{
We consider spline functions over simplicial meshes in $\RR^n$.
We assume that the spline pieces
join together with some finite order of smoothness
but the pieces themselves are infinitely smooth.
Such splines can have extra orders of smoothness at a vertex,
a property known as \emph{supersmoothness}, which plays a
role in the construction of multivariate splines
and in the finite element method.
In this paper we characterize supersmoothness
in terms of the degeneracy of spaces of polynomial splines
over the cell of simplices sharing the vertex,
and use it to determine the maximal order of supersmoothness of
various cell configurations.
}

\smallskip

\noindent {\em Math Subject Classification: }
Primary: 41A15, 65D07, % spline approximation, splines
Secondary: 41A58, 65N30 % Taylor expansions, Finite elements

\smallskip


\smallskip

%%%%%%%%%%%%%%%%%%%%%%%%%%%%%%%%%%%%%%%%%%%%%%%%%%%%%%%%%%%%%%%%
\section{Introduction}
%%%%%%%%%%%%%%%%%%%%%%%%%%%%%%%%%%%%%%%%%%%%%%%%%%%%%%%%%%%%%%%%
Polynomial splines over a simplicial partition of a domain in $\RR^n$
(a triangular mesh in 2D, a tetrahedral mesh in 3D, and so on)
are functions whose pieces are polynomials up to a certain degree $d$
and which join together with
some order of continuity~$r$.
Such spline functions may have extra orders of smoothness at
a vertex of the mesh, a property known as {\it supersmoothness}
as suggested by Sorokina~\cite{sorokina2010intrinsic}.
For example, the Clough-Tocher macroelement,
which is $C^1$ piecewise cubic,
is twice differentiable at the refinement point, as first observed
by Farin \cite{farin1980},
and so this element can be said to have supersmoothness of order~2
at that point.

For the construction of splines or finite elements with higher orders
of continuity, it is important to recognize and
make use of supersmoothness.
For example, it plays a role in many of the
macroelement constructions surveyed by Lai and Schumaker\cite{lai2007spline},
where applications of splines to
approximation theory and computer-aided geometric design are discussed.
The concept of supersmoothness is also relevant to
the finite element method.
Motivated by structure-preserving or compatible discretizations
there has recently been an increased interest in investigating the use of
splines for vector fields and differential complexes~\cite{alfeld2016linear,
arnold2006finite,christiansen2018generalized,fu2018exact}.
The de~Rham complex reveals a connection between smooth,
e.g., $C^1$, finite elements and the Stokes problem in fluid mechanics.
In a discrete de~Rham complex, the spline spaces for the velocity field
may inherit the supersmoothness of the scalar field,
\cite{alfeld2016linear,christiansen2018generalized,
fu2018exact,schumaker1989super}.
Thus, supersmoothness is also of importance
in the study of these problems. 

Since Farin's observation about the Clough-Tocher element,
Sorokina, in \cite{sorokina2010intrinsic} and
\cite{sorokina2014redundancy} has derived further
supersmoothness properties of polynomial splines,
and in particular higher order supersmoothness in a cell in 2D;
see equation (\ref{eq:rho2D}).
More recently, Shekhtman and Sorokina~\cite{shekhtman2015intrinsic}
observed that supersmoothness is a phenomenon of
more general spline functions, not only piecewise polynomials.
Their results imply that at the vertex $v$ of a triangulation
with $m$ incoming edges all having different slopes, any
$C^r$ spline with $r \ge m-2$ has derivatives of order $r+1$ at $v$
as long as the spline pieces themselves have
$C^{r+1}$ continuity in a neighbourhood of $v$.

The results of~\cite{shekhtman2015intrinsic} were the motivation
for this paper. If we simplify the framework of~\cite{shekhtman2015intrinsic}
and assume that all the spline pieces are
$C^\infty$ smooth, which is the case for polynomials and
many other functions of interest,
can we extend the
results to higher orders of supersmoothness and also to
higher Euclidean space dimensions?
Our solution is to simplify the problem
by deriving a characterization of
supersmoothness in terms of the degeneracy of
polynomial spline spaces over the cell (in Theorem~\ref{thm:charac}).
Using this, the maximal order of supersmoothness at a vertex
can be determined once a general formula
for the dimensions of the polynomial spline spaces over the cell is known.
At the end of the paper we apply these results to various cell configurations.

%%%%%%%%%%%%%%%%%%%%%%%%%%%%%%%%%%%%%%%%%%%%%%%%%%%%%%%%%%%%%%%%
\section{Cells and supersmoothness}\label{sec:ss}
%%%%%%%%%%%%%%%%%%%%%%%%%%%%%%%%%%%%%%%%%%%%%%%%%%%%%%%%%%%%%%%%
We start with some definitions.
\subsection{Simplicial meshes and cells}
Let $\Delta$ be a set of $n$-simplices in $\RR^n$.
We call $\Delta$ a mesh if the intersection between
any two n-simplices $T_1,T_2 \in \Delta$
is either empty or a common $k$-face for some $k$, $0 \le k \le n-1$.
We let $\Omega = \cup \{T : T \in \Delta \}$.

If $v$ is a vertex in the mesh,
we denote by $\Delta_v \subset \Delta$ the n-simplices
in $\Delta$ that share $v$, and we call $\Delta_v$ a cell.
Let $\Omega_v = \cup \{T : T \in \Delta_v \}$.
We will say that $v$ is an interior vertex of $\Delta$ if
$v$ is in the interior of $\Omega_v$,
in which case we will say that $\Delta_v$ is an interior cell.

In 2D an interior cell $\Delta_v$ is a sequence
of triangles $\Delta_v = \{T_1,T_2,\ldots,T_m\}$, $m \ge 3$,
that form a star-shaped polygon $\Omega$,
as in Figure~\ref{fig:cells2D}.
\begin{figure}[ht]
\centering
\includegraphics[height=0.17\textheight]{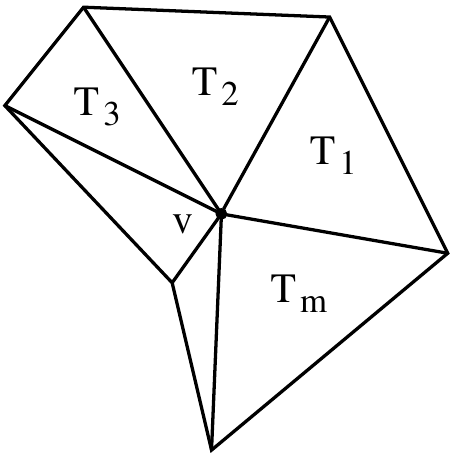}
\qquad \qquad
\includegraphics[height=0.14\textheight]{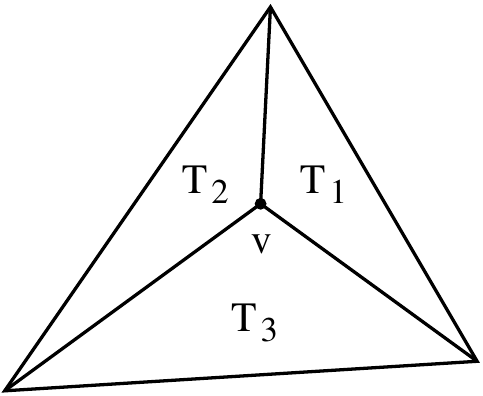}
\caption{Cells in 2D.}
\label{fig:cells2D}
\end{figure}
In the special case that $m=3$, $\Delta_v$
is known as a Clough-Tocher split since it can also be constructed
by refinement.
We could start with any triangle $T$ in the plane
(the outer triangle in the figure), then let $v$
be any point inside $T$ and connect the three edges of $T$ to $v$,
thus creating three sub-triangles of $T$.

In 3D a cell is a collection of tetrahedra.
A simple example of an interior cell $\Delta_v$ in 3D
is the Alfeld split, constructed by
choosing a tetrahedron $T$, then any point $v$ inside $T$ and
connecting~$v$ to the four triangular faces of $T$.
The resulting cell has four tetrahedra,
as in Figure~\ref{fig:Alfeld}.
\begin{figure}[ht]
\centering
\includegraphics[height=0.17\textheight]{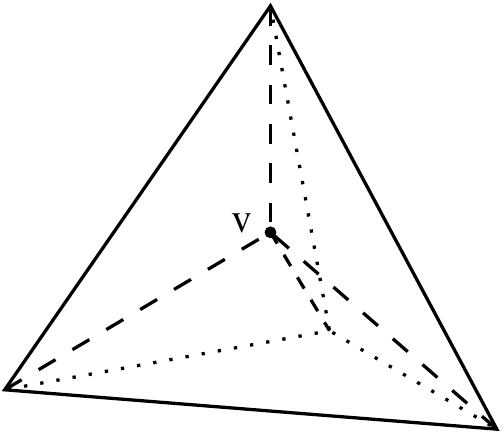}
\caption{Alfeld split in 3D.}
\label{fig:Alfeld}
\end{figure}

\subsection{Splines}
In this paper in order to have a notion of supersmoothness
of various orders we need to view a spline as a set of
pairs of open neighbourhoods and smooth functions,
one pair for each n-simplex in the mesh $\Delta$.
Thus a spline $\sigma$ has the form
$$ \sigma = \{(U_T,f_T) : T \subset U_T \subset \RR^n, \,
      \hbox{$U_T$ open}, \, f_T \in C^\infty(U_T),
                  \, T \in \Delta \}, $$
and we denote by $S(\Delta)$ the set of all such splines.
The pieces $f_T$ could, for example, be
polynomials of any degree (in which case we can take $U_T = \RR^n$),
trigonometric functions, rational functions, and so on.

Next we consider how the pieces of $\sigma$ might fit together.
Let $f \in C^\infty(U)$ for some open set $U \subset \RR^n$
and let $x = (x_1,x_2,\ldots,x_n) \in U$.
Let $\alpha = (\alpha_1,\ldots,\alpha_n)$ be a multi-index, with
$\alpha_1, \ldots, \alpha_n \ge 0$.
Then we denote by
$$ D^\alpha f(x) = 
     \left(\frac{\partial}{\partial x_1}\right)^{\alpha_1} \cdots
     \left(\frac{\partial}{\partial x_n}\right)^{\alpha_n} f(x), $$
the corresponding partial derivative of $f$ at $x$, of order
$|\alpha| = \alpha_1 + \cdots + \alpha_n$.
We will say that a spline $\sigma \in S(\Delta)$ has
smoothness $r \ge 0$ if
$$ D^\alpha f_{T_1}|_F = D^\alpha f_{T_2}|_F, \quad
                 |\alpha| \le r, \quad
              T_1, T_2 \in \Delta, \quad
              T_1 \cap T_2 = F,
     \quad \hbox{$F$ is an $(n-1)$-face}. $$
We will denote by $S^r(\Delta)$ the set of all such splines.

\subsection{Smoothness at a vertex}
Suppose that $v$ is a vertex of $\Delta$ and let $\sigma \in S(\Delta_v)$.
We will say that $\sigma$ has smoothness of order $\rho$ at $v$ if
$$ D^\alpha f_{T_1}(v) = D^\alpha f_{T_2}(v), \quad
                 |\alpha| \le \rho, \quad
              T_1,T_2 \in \Delta_v. $$
We will denote by $C^\rho(\Delta_v)$ the set of all such splines.

\subsection{Supersmoothness}
Now we look at enhanced smoothness of splines at an interior vertex $v$
of $\Delta$. We will say that a spline $\sigma\in S^r(\Delta_v)$ has
\emph{supersmoothness of order $\rho \ge r$ at $v$} if
$\sigma \in C^\rho(\Delta_v)$.
Thus we are interested in the question of whether
$S^r(\Delta_v) \subset C^\rho(\Delta_v)$ for some
$\rho > r$.
This will depend on the geometric configuration of the $n$-simplices
of $\Delta_v$.

%%%%%%%%%%%%%%%%%%%%%%%%%%%%%%%%%%%%%%%%%%%%%%%%%%%%%%%%%%%%%%%%
\section{Taylor approximations}\label{sec:Taylor}
%%%%%%%%%%%%%%%%%%%%%%%%%%%%%%%%%%%%%%%%%%%%%%%%%%%%%%%%%%%%%%%%
Our aim is to characterize supersmoothness in terms of the
degeneracy of polynomial splines.
The first step in the derivation is to study Taylor approximations.
Let $f \in C^\infty(B)$ for some domain $B \subset \RR^n$.
%and let $\alpha = (\alpha_1,\ldots,\alpha_n)$ be a multi-index, with
%$\alpha_1, \ldots, \alpha_n \ge 0$.
%We denote the corresponding partial derivative of $f$ by
%$$ D^\alpha f(x) = 
%     \left(\frac{\partial}{\partial x_1}\right)^{\alpha_1} \cdots
%     \left(\frac{\partial}{\partial x_n}\right)^{\alpha_n} f(x), 
%   \quad x = (x_1,\ldots,x_n) \in B. $$
With respect to a point $v=(v_1,\ldots,v_n) \in B$, we denote
the Taylor approximation of $f$ of order $\rho \ge 0$
by
$$ \calT_{v,\rho} f(x) = \sum_{|\alpha| \le \rho}
     \frac{D^\alpha f(v)}{\alpha!} 
             (x_1-v_1)^{\alpha_1} \cdots (x_n-v_n)^{\alpha_n},
    \qquad x \in B, $$
where $\alpha! = \alpha_1! \cdots \alpha_n!$.
We will make use of the following property of these Taylor approximations.
\begin{lemma}\label{lem:Taylor1}
Let $v,w$ be distinct points in $\RR^n$ and
let $e = [v,w]$ be the line segment connecting them.
Let $B \subset \RR^n$ be some domain containing $e$.
Suppose that $f, g \in C^\infty(B)$ and that
$f|_e = g|_e$.
Then, for any $\rho \ge 0$,
$\calT_{v,\rho}f|_e = \calT_{v,\rho}g|_e$.
\end{lemma}

\begin{proof}
We can represent the line segment $e$ parametrically as
$$ e = \{v + t u : 0 \le t \le 1\}, $$
where $u = w - v$.
Letting $x = v + tu$ for some $t \in [0,1]$, we find that
\begin{align*}
 \calT_{v,\rho}f(x)
  = \calT_{v,\rho} f(v+tu)
  & = \sum_{|\alpha| \le \rho}
     \frac{D^\alpha f(v)}{\alpha!} 
             (tu_1)^{\alpha_1} \cdots (tu_n)^{\alpha_n} \cr
  & = \sum_{i=0}^\rho \frac{t^i}{i!} 
      \sum_{|\alpha| = i}
     \frac{i!}{\alpha!} 
        D^\alpha f(v)
             u_1^{\alpha_1} \cdots u_n^{\alpha_n}
  = \sum_{i=0}^\rho \frac{t^i}{i!} 
      h^{(i)}(0),
\end{align*}
where
$$ h(\tau) = f(v + \tau u), \quad \tau \in [0,1]. $$
Since $f$ and $g$ are equal on $e$, we also have
$$ h(\tau) = g(v + \tau u), \quad \tau \in [0,1], $$
and so
$$ \calT_{v,\rho}g(x) = \sum_{i=0}^\rho \frac{t^i}{i!} 
      h^{(i)}(0) = \calT_{v,\rho}f(x). $$
\end{proof}

We want to generalize this property to derivatives of $f$ and $g$.
To do this we first show
\begin{lemma}\label{lem:Taylor2}
Let $v \in \RR^n$ and suppose
$f \in C^\infty(B)$ for some domain $B \subset \RR^n$ containing~$v$.
Then, for any integer $\rho \ge 0$ and any
multi-index $\beta$ with $|\beta| \le \rho$,
$$ D^\beta \calT_{v,\rho} f
  = \calT_{v,\rho - |\beta|} D^\beta f. $$
\end{lemma}

\begin{proof}
From the definition of $\calT_{v,\rho}$, for $x \in B$,
\begin{align*}
 D^\beta \calT_{v,\rho} f(x)
     &= \sum_{\substack{|\alpha| \le \rho \\ \alpha \ge \beta}}
     \frac{D^\alpha f(v)}{(\alpha-\beta)!} 
             (x_1-v_1)^{\alpha_1 - \beta_1} 
           \cdots (x_n-v_n)^{\alpha_n - \beta_n} \cr
     &= \sum_{|\alpha| \le \rho - |\beta|}
     \frac{D^{\alpha+\beta} f(v)}{\alpha!} 
             (x_1-v_1)^{\alpha_1} 
           \cdots (x_n-v_n)^{\alpha_n}
   = \calT_{v,\rho - |\beta|} D^\beta f(x).
\end{align*}
\end{proof}

From Lemmas~\ref{lem:Taylor1} and~\ref{lem:Taylor2} we obtain
\begin{lemma}\label{lem:Taylor3}
Let $v,w,e,B$ be as in Lemma~\ref{lem:Taylor1}.
Suppose that $f, g \in C^\infty(B)$ and that for some $r \ge 0$,
\begin{equation}\label{eq:fequalsg}
 D^\beta f|_e = D^\beta g|_e, \qquad |\beta| \le r.
\end{equation}
Then, for any $\rho \ge 0$,
\begin{equation}\label{eq:TfequalsTg}
 D^\beta \calT_{v,\rho}f|_e = D^\beta \calT_{v,\rho}g|_e, \qquad |\beta| \le r.
\end{equation}
\end{lemma}

\begin{proof}
If $|\beta| > \rho$, equation (\ref{eq:TfequalsTg}) trivially holds
since both sides are equal to $0$.
If $|\beta| \le \rho$, by Lemma~\ref{lem:Taylor2},
equation (\ref{eq:TfequalsTg}) is equivalent to
$$
 \calT_{v,\rho-|\beta|} D^\beta f|_e = 
    \calT_{v,\rho-|\beta|} D^\beta g|_e,
$$
and by Lemma~\ref{lem:Taylor1}, this is implied by
equation (\ref{eq:fequalsg}).
\end{proof}

%%%%%%%%%%%%%%%%%%%%%%%%%%%%%%%%%%%%%%%%%%%%%%%%%%%%%%%%%%%%%%%%
\section{Characterization of supersmoothness}\label{sec:charac}
%%%%%%%%%%%%%%%%%%%%%%%%%%%%%%%%%%%%%%%%%%%%%%%%%%%%%%%%%%%%%%%%
We are now approaching a characterization of supersmoothness.
\subsection{Polynomial spline spaces}
For integers $r$ and $d$ with $0 \le r \le d$ let
$$
S_d^r(\Delta) := \{s\in C^r(\Omega): s|_T \in \Pi_d, \,\, T \in \Delta\},
$$
where~$\Pi_d$ is the linear space of polynomials in $\RR^n$
of degree at most $d$.
Thus $S_d^r(\Delta)$ is the usual linear space of polynomial splines
on $\Delta$ of smoothness $r$ and degree at most~$d$.

\subsection{Degeneracy}
Consider an interior cell $\Delta_v$.
By definition, for any $r \ge 0$ we have $\Pi_d \subset S_d^r(\Delta_v)$.
Sometimes, however, depending on $\Delta_v$ and $r$, we might have
$S_d^r(\Delta_v) = \Pi_d$.
In this case $S_d^r(\Delta_v)$ contains no `true' splines, only polynomials,
and we view $S_d^r(\Delta_v)$ as being degenerate in this sense.
\begin{definition}
We will say that
$S_d^r(\Delta_v)$ is \emph{degenerate} if $S_d^r(\Delta_v) = \Pi_d$.
\end{definition}

As an example, the space $S_r^r(\Delta_v)$ is degenerate for any
$r \ge 0$.

\subsection{Piecewise Taylor approximations}
Next recall the more general set of splines
$S(\Delta_v)$ and let
$$ \sigma = \{(U_T,f_T): T \in \Delta_v \} \in S(\Delta_v). $$
For any~$\rho \ge 0$, we can make the following
piecewise Taylor approximation of $\sigma$:
$$ \calT_{v,\rho}\sigma :=
    \{(\RR^n,\calT_{v,\rho}f_T) : T \in \Delta_v \} \in S(\Delta_v). $$
Due to Lemma~\ref{lem:Taylor3} we next show
\begin{lemma}\label{lem:Taylor4}
If $\sigma \in S^r(\Delta_v)$ for any $r \ge 0$
then $\calT_{v,\rho}\sigma \in S^r(\Delta_v)$ for any $\rho \ge 0$.
\end{lemma}

\begin{proof}
Let $T_1,T_2 \in \Delta_v$ be two $n$-simplices
that share a common $(n-1)$-face $F$.
The face $F$ is the union of all the line segments $e$ that
connect $v$ to the $(n-2)$-dimensional face of $F$ opposite to $v$.
The pieces $f_{T_1}$ and $f_{T_2}$ have
the same derivatives up to order $r$ on $e$.
Therefore, by Lemma~\ref{lem:Taylor3},
the two Taylor approximations
$\calT_{v,\rho}f_{T_1}$ and $\calT_{v,\rho}f_{T_2}$ have
the same derivatives up to order $r$ on $e$.
Therefore, they have the same
derivatives up to order~$r$ on the whole face~$F$.
Thus $\calT_{v,\rho}\sigma \in S^r(\Delta_v)$ as claimed.
\end{proof}

\subsection{Characterization}
With the previous definitions in place the characterization is as follows.
\begin{theorem}\label{thm:charac}
Let $\Delta_v$ be an interior cell and suppose
$0 \le r \le \rho$. Then $S^r(\Delta_v) \subset C^\rho(\Delta_v)$
if and only if $S_\rho^r(\Delta_v)$ is degenerate.
\end{theorem}

\begin{proof}
Suppose that $S_\rho^r(\Delta_v)$ is degenerate and let
$$ \sigma = \{(U_T,f_T): T \in \Delta_v \} \in S^r(\Delta_v). $$
By Lemma~\ref{lem:Taylor4},
$\calT_{v,\rho}\sigma \in S^r(\Delta_v)$.
Therefore we can define a polynomial spline $s \in S_\rho^r(\Delta_v)$, by
$$ s|_T = \calT_{v,\rho} f_T, \quad T \in \Delta_v. $$
By the assumption that
$S_\rho^r(\Delta_v)$ is degenerate,
$s \in \Pi_\rho$.
Then, for any  $T_1,T_2 \in \Delta_v$,
$$ D^\alpha f_{T_1}(v) = D^\alpha \calT_{v,\rho}f_{T_1}(v)
     = D^\alpha s(v) = D^\alpha \calT_{v,\rho}f_{T_2}(v) = 
                D^\alpha f_{T_2}(v), \quad |\alpha| \le \rho. $$
This proves that $S^r(\Delta_v) \subset C^\rho(\Delta_v)$.

Conversely, suppose that
$S^r(\Delta_v) \subset C^\rho(\Delta_v)$ and let
$s \in S_\rho^r(\Delta_v)$.
Then we can define
$$ \sigma = \{(\RR^n,s|_T) : T \in \Delta_v \} \in S^r(\Delta_v). $$
The assumption that $S^r(\Delta_v) \subset C^\rho(\Delta_v)$
implies that $\sigma \in C^\rho(\Delta_v)$.
Therefore, for any $T_1,T_2 \in \Delta_v$,
$$ D^\alpha s|_{T_1}(v) = D^\alpha s|_{T_2}(v), \quad |\alpha| \le \rho. $$
Since $s|_{T_1}, s|_{T_2} \in \Pi_\rho$, this implies that
$s|_{T_1} = s|_{T_2}$.
Thus $s \in \Pi_\rho$. This proves that $S_\rho^r(\Delta_v)$ is degenerate.
%$$ D^\alpha s_1(v) = D^\alpha s_2(v), \quad |\alpha| \le \rho, $$
%where $s_i := s|_{T_i}$, $i=1,2$.
%Since $s_1, s_2 \in \Pi_\rho$, this implies that $s_1 = s_2$.
%Thus $s \in \Pi_\rho$. This proves that $S_\rho^r(\Delta_v)$ is degenerate.
\end{proof}

We remark that this theorem also holds if we reduce the 
smoothness assumption on the pieces $f_T$ of 
the splines $\sigma$ in $S^r(\Delta_v)$
to being in $C^{\rho}(U_T)$ instead of in $C^\infty(U_T)$.

\subsection{Maximal order of supersmoothness}
We can also consider the \emph{mos} (maximal order of supersmoothness)
of $S^r(\Delta_v)$, i.e.,
$$ \mos S^r(\Delta_v) := 
    \max \{\rho \ge r : S^r(\Delta_v) \subset C^\rho(\Delta_v) \}. $$
To characterize this, observe that
we have a nested sequence of spaces,
$$ \Pi_r = S_r^r(\Delta_v) \subset S_{r+1}^r(\Delta_v) \subset
    S_{r+2}^r(\Delta_v) \subset \cdots. $$
Therefore, if $S_d^r(\Delta_v)$ is non-degenerate for some $d \ge r$, then
$S_k^r(\Delta_v)$ is non-degenerate for all~$k \ge d$.
Thus, for any cell $\Delta_v$ and any $r \ge 0$,
there is a unique highest degree
$d \ge r$ such that $S_d^r(\Delta_v)$ is degenerate.
From Theorem~\ref{thm:charac} we deduce
\begin{corollary}\label{cor:mos}
$\mos S^r(\Delta_v) =  \max \{d \ge r : 
                    \hbox{$S_d^r(\Delta_v)$ is degenerate} \}$.
\end{corollary}

%%%%%%%%%%%%%%%%%%%%%%%%%%%%%%%%%%%%%%%%%%%%%%%%%%%%%%%%%%%%%%%%
\section{Applications}\label{sec:applications}
%%%%%%%%%%%%%%%%%%%%%%%%%%%%%%%%%%%%%%%%%%%%%%%%%%%%%%%%%%%%%%%%

We now apply the characterization theorem to some
concrete examples. For a cell~$\Delta_v$ in $\RR^n$
and smoothness $r \ge 0$ the spline space
$S_d^r(\Delta_v)$, with $d \ge r$, is degenerate if
\begin{equation}\label{eq:compute}
 \dim S_d^r(\Delta_v) = \dim \Pi_d = \binom{d+n}{n}.
\end{equation}
For some cell configurations degeneracy is known
for specific degrees $d > r$.
We then conclude from Theorem~\ref{thm:charac} that
all splines in $S^r(\Delta_v)$ have supersmoothness of order $d$,
but we do not know whether $d$ is optimal.
However, if we know the dimensions of
all the spaces $S_d^r(\Delta_v)$, $d > r$,
we obtain the maximal supersmoothness from Corollary~\ref{cor:mos}
by finding the largest $d$ satisfying (\ref{eq:compute}).

We note also that Alfeld \cite{mds} has computed the dimension of
many spline spaces over various kinds of cell.
These computational results also determine supersmoothness
by Theorem \ref{thm:charac} or Corollary~\ref{cor:mos}.

%%%%%%%%%%%%%%%%%%%%%%%%%%%%%%%%%%%%%%%%%%%%%%%%%%%%%%%%%%%%%%%%
\subsection{Clough-Tocher split}
%%%%%%%%%%%%%%%%%%%%%%%%%%%%%%%%%%%%%%%%%%%%%%%%%%%%%%%%%%%%%%%%
In $\RR^2$, when $\Delta_v$ has three triangles it is a Clough-Tocher split,
$\Delta_{CT}$, and,
using the theory of Bernstein-B{\'e}zier polynomials,
Farin showed in~\cite[Theorem~7]{farin1980}
that $S_{r+1}^r(\Delta_{CT})$ is degenerate for any $r\ge 1$.
He then concluded in ~\cite[Corollary~8]{farin1980}
that the pieces of any spline in $S_d^r(\Delta_{CT})$, $1 \le r \le d$,
have matching derivatives of order $r+1$ at $v$.

We can now apply Theorem~\ref{thm:charac}
to conclude more generally that
$S^r(\Delta_{CT}) \subset C^{r+1}(\Delta_{CT})$
for $r \ge 1$.
However, this is not optimal supersmoothness for general~$r$.

%%%%%%%%%%%%%%%%%%%%%%%%%%%%%%%%%%%%%%%%%%%%%%%%%%%%%%%%%%%%%%%%
\subsection{An arbitrary cell in 2D}
%%%%%%%%%%%%%%%%%%%%%%%%%%%%%%%%%%%%%%%%%%%%%%%%%%%%%%%%%%%%%%%%
Sorokina made a substantial generalization of Farin's result.
She showed in~\cite[Theorem 3.1]{sorokina2010intrinsic} that
if $\Delta_v$ has $m$ triangles, and the
$m$ interior edges have different slopes, then
for $0 \le r \le d$, the pieces of
any spline $s \in S_d^r(\Delta_v)$ have matching derivatives
at $v$ up to order
\begin{equation}\label{eq:rho2D}
 \rho = r + \left\lfloor \frac{r+1}{m-1} \right\rfloor.
\end{equation}
The proof was based on comparing the dimension of
$S_d^r(\Delta)$ with those of superspline spaces.
Since $\rho$ in (\ref{eq:rho2D})
is independent of the degree $d$, one might expect a more general result.
This was also suggested by the work of
Shekhtman and Sorokina~\cite{shekhtman2015intrinsic}.
From (\ref{eq:rho2D}) it follows that there is
at least one order of supersmoothness when $r \ge m-2$
for any degree $d \ge r$.
Shekhtman and Sorokina showed that this is also true
for more general splines, in other words, in our notation,
$S^r(\Delta_v) \subset C^{r+1}(\Delta_v)$ when $r \ge m-2$.
Their proof was based on expressing partial derivatives
as linear combinations of directional derivatives
along the edges meeting at $v$.
Using Corollary~\ref{cor:mos}, we can now improve this result
to match that of the polynomial case.
To do this, we first transform the dimension formula
of Lai and Schumaker~\cite{lai2007spline}
into a more suitable form.

\begin{lemma}\label{lem:celldim}
Suppose $\Delta_v$ has $m$ triangles and suppose
there are $m_v$ different slopes among the interior edges of $\Delta_v$.
For $0 \le r \le d $,
\begin{equation}\label{eq:celldim}
 \dim S_d^r(\Delta_v)
  = \dim \Pi_{d} + (m-m_v) \dim \Pi_{d-r-1}
  + \sum_{j=1}^{d-r} (\tau_{v, j})_+,
\end{equation}
where $\tau_{v,j} := j(m_v - 1)-(r+1)$, and
$(x)_{+}:=x$ if $x > 0$ and $(x)_{+}:=0$ otherwise.
\end{lemma}

\begin{proof}
The dimension of $S_{d}^r(\Delta_v)$ was derived in
 \cite[Theorem 9.3]{lai2007spline} in the form
\begin{equation}\label{eq:celldimLS}
 \dim S_d^r(\Delta_v)
  = \binom{r+2}{2} + m \binom{d-r+1}{2}
  + \sum_{j=1}^{d-r} (-\tau_{v,j})_+.
\end{equation}
Using the fact that
$$ \sum_{j=1}^{d-r} \tau_{v,j} =
 m_v \sum_{j=1}^{d-r} j - \sum_{j=1}^{d-r} (r+j+1)
  = m_v \binom{d-r+1}{2} 
  - \binom{d+2}{2} + \binom{r+2}{2}, $$
we can rewrite (\ref{eq:celldimLS}) as
$$
 \dim S_d^r(\Delta_v)
  = \dim \Pi_{d} + (m-m_v) \binom{d-r+1}{2}
  + \sum_{j=1}^{d-r} \big((-\tau_{v,j})_+ + \tau_{v,j}\big).
$$
Then, using the fact that
$(-x)_{+} + x = x_+$, the result follows.
\end{proof}

\begin{theorem}\label{thm:one}
Suppose $\Delta_v$ has $m$ triangles and suppose
there are $m_v$ different slopes among the interior edges of $\Delta_v$.
Then for $r \ge 0$,
\begin{equation}\label{eq:mos2D}
 \mos S^r(\Delta_v) = \begin{cases}
             r + \left\lfloor \frac{r+1}{m-1} \right\rfloor,
                 & m_v = m; \cr
             r, & m_v < m.
              \end{cases}
\end{equation}
\end{theorem}

\begin{proof}
By Corollary~\ref{cor:mos},
it is sufficient to determine the highest degree $d\geq r$
such that $S_d^{r}(\Delta_v)$ is degenerate, i.e.,
such that $\dim S_{d}^r(\Delta_v) = \dim \Pi_d$.
To do this we use Lemma~\ref{lem:celldim}.
Suppose $m_v < m$.
If $d = r+1$, the second term in (\ref{eq:celldim}) is
strictly positive and so $S_{r+1}^r(\Delta_v)$ is non-degenerate.
Therefore $S_d^r(\Delta_v)$ is degenerate if and only if $d=r$.
Otherwise, $m_v = m$.
Then considering the third term in (\ref{eq:celldim}),
$S_d^r(\Delta_v)$ is degenerate if and only if $\tau_{v,j} \le 0$
for all $j=1,\ldots, d-r$,
or equivalently $\tau_{v,d-r} \le 0$,
which is equivalent to
$$ d \le r + \left\lfloor \frac{r+1}{m-1} \right\rfloor. $$
\end{proof}

As an example, for the Clough-Tocher split we have $m=m_v = 3$
and so
\begin{equation}\label{eq:CT}
 \mos S^r(\Delta_{CT}) = r + \left\lfloor \frac{r+1}{2} \right\rfloor.
\end{equation}

%%%%%%%%%%%%%%%%%%%%%%%%%%%%%%%%%%%%%%%%%%%%%%%%%%%%%%%%%%%%%
\subsection{The Alfeld split in $\RR^n$}
%%%%%%%%%%%%%%%%%%%%%%%%%%%%%%%%%%%%%%%%%%%%%%%%%%%%%%%%%%%%%
The dimensions of the spaces $S_d^r(\Delta_v)$ are not
currently known for a general cell $\Delta_v$ in~$\RR^n$ for $n \ge 3$.
However, they are known in special cases.
One of these is the Alfeld split in $\RR^n$.
In $\RR^n$, $n\geq 2$, the split is constructed
by choosing any $n$-dimensional simplex~$T$ and splitting it into
$n+1$ smaller simplices by choosing an arbitrary interior point $v$ in $T$
and connecting it to each of the $n+1$ faces (of dimension $n-1$) of $T$.
We denote this split by $\Delta_{A,n}$.
The 3D case $\Delta_{A,3}$ is shown in Figure~\ref{fig:Alfeld}.

Using the theory of Bernstein-B{\'e}zier polynomials,
Worsey and Farin showed in~\cite[Lemma 3.1]{worsey1987ann}
that $S_2^1(\Delta_{A,n})$ is degenerate.
From this, Theorem~\ref{thm:charac} implies that
$S^1(\Delta_{A,n}) \subset C^2(\Delta_{A,n})$.
But we can make a further generalization by invoking
the recently derived dimension formula of
Foucart and Sorokina~\cite{foucart2013generating}
and Schenck~\cite{schenck2014splines}.
Let us define, for $n \ge 1$ and $r \ge 0$,
$$ \rho_{n,r} := r + (n-1) \left\lfloor \frac{r+1}{2} \right\rfloor. $$

\begin{theorem}\label{thm:dim-Alfeld}
The maximal order of supersmoothness of the Alfeld split is
$$ \mos S^{r}(\Delta_{A,n}) = \rho_{n,r}. $$
\end{theorem}

\begin{proof}
The dimensions of the polynomial spline spaces on the Alfeld split
were generated and conjectured
by Foucart and Sorokina~\cite{foucart2013generating}
and proved by Schenck~\cite{schenck2014splines}:
for $0 \le r \le d$,
$$
\dim S_d^r(\Delta_{A,n}) = \dim \Pi_d + A(n, d, r),
$$
where 
$$
A(n, d, r)=\begin{cases}
n\binom{d+n-(r+1)(n+1)/2}{n}, \quad \mbox{if $r$ is odd};\\
\binom{d+n-1-r(n+1)/2}{n}+\cdots +\binom{d-r(n+1)/2}{n},
    \quad \mbox{if $r$ is even}.
\end{cases}
$$
Therefore, $S_d^r(\Delta_{A,n})$ is degenerate if and only if $A(n,d,r) = 0$,
or equivalently if
$$
\begin{cases}
{d-(r+1)(n+1)/2\leq -1}, \quad \mbox{if $r$ is odd};\\
d-1-r(n+1)/2\leq -1, \quad \mbox{if $r$ is even}.
\end{cases}
$$
By Corollary~\ref{cor:mos}, the maximal order of supersmoothness
is the largest such $d$, i.e.,
$$
d =
\begin{cases}
(r+1)(n+1)/2 -1, \quad \mbox{if $r$ is odd};\\
r(n+1)/2, \quad \mbox{if $r$ is even},
\end{cases}
$$
or equivalently, $d = \rho_{n,r}$.
\end{proof}

\begin{comment}
As special cases of Theorem \ref{thm:dim-Alfeld},
for the Clough-Tocher macroelements in 2D,
$$
\mos S^r(\Delta_A)=
\begin{cases}
3(r+1)/2 -1, \quad \mbox{if $r$ is odd};\\
3r/2, \quad \mbox{if $r$ is even}.
\end{cases}
$$
For the Alfeld macroelements in 3D,
$$
\mos S^r(\Delta_A)
=\begin{cases}
2r+1, \quad \mbox{if $r$ is odd};\\
2r, \quad \mbox{if $r$ is even},
\end{cases}
$$
\end{comment}

For example, $S^1(\Delta_{A,n}) \subset C^n(\Delta_{A,n})$, and
in particular, $S^1(\Delta_{A,3}) \subset C^3(\Delta_{A,n})$,
which shows that the $C^1$
macro-element on the Alfeld split $\Delta_{A,3}$
described in \cite[Section 18.3]{lai2007spline}
has supersmoothness of order $3$.

We note that Theorem~\ref{thm:dim-Alfeld} in the case $n=2$
agrees with the supersmoothness of
the Clough-Tocher split in equation (\ref{eq:CT}).

%%%%%%%%%%%%%%%%%%%%%%%%%%%%%%%%%%%%%%%%%%%%%%%%%%%%%%
\subsection{The $\Delta_{k,n}$ split}
%%%%%%%%%%%%%%%%%%%%%%%%%%%%%%%%%%%%%%%%%%%%%%%%%%%%%%
Worsey and Farin $\cite{worsey1987ann}$ proposed an alternative
generalization of the Clough-Tocher split to $\RR^n$,
using recursion through the
Euclidean dimensions; see also \cite{sorokina2009multivariate}.
To split an $n$-simplex $T$,
they first split the faces of $T$ of dimension 2 (triangles)
by making a Clough-Tocher split.
They next split each 3-face (a tetrahedron) $F$ of $T$ by
choosing any point in the relative interior of $F$
and connecting it to the twelve triangles on the boundary of $F$
constructed in the previous step.
They continue in a similar way, next splitting faces of $T$ of dimension 4
and so on.
Part of a Worsey-Farin split in 3D is shown in Figure \ref{fig:worsey-farin},
viewed as a refinement of an Alfeld split.
One of the subsimplices of the Alfeld split has
been split into three.
\begin{figure}[ht]
\centering
\includegraphics[height=0.17\textheight]{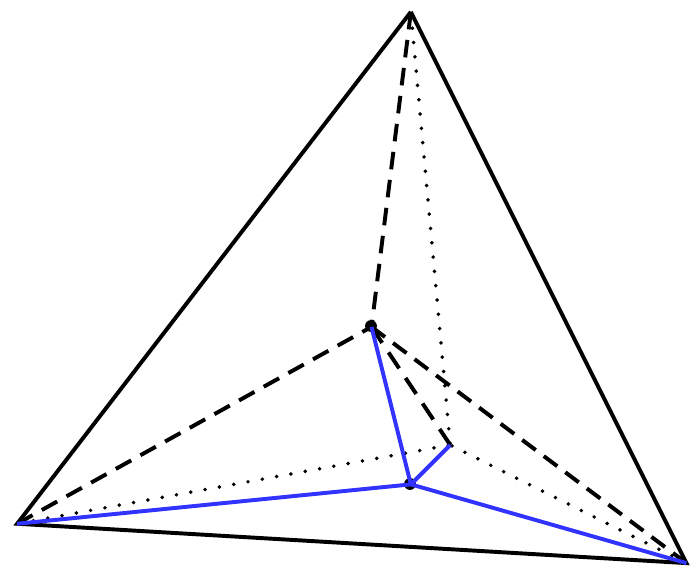}
\caption{Part of a Worsey-Farin split in 3D.}
\label{fig:worsey-farin}
\end{figure}

Let us consider a more general split.
We choose any Euclidean dimension $k$, $1 \le k \le n$.
We then initialize the splitting by splitting each $k$-face $F$ of $T$
by choosing any point in the relative interior of $F$
and connecting it to the $(k-1)$-faces of $F$.
Then, for $j=k+1,\ldots,n$ in sequence,
we split each $j$-face $F$ of $T$
by choosing any point in the relative interior of $F$
and connecting it to the
$$ (j+1) \times \frac{j!}{k!} = \frac{(j+1)!}{k!} $$
simplices of dimension $(j-1)$
on the boundary of $F$ constructed in the previous step.
The resulting split of $T$ is a cell
around the point $v$ in the interior of~$T$ chosen at the last
step ($j=n$). It has $(n+1)!/k!$ sub-simplices and we denote
it by~$\Delta_{k,n}$.

For example, in 2D, $\Delta_{2,2}$ is a Clough-Tocher split
and $\Delta_{1,2}$ is a Powell-Sabin 6-split.
In 3D, $\Delta_{3,3}$ is an Alfeld split,
$\Delta_{2,3}$ is a Worsey-Farin split
and $\Delta_{1,3}$ is a Worsey-Piper split.

By construction, each of the $(n-1)$-faces of $T$ is itself
split into a $\Delta_{k,n-1}$ split.
A split $\Delta_{k,n}$, $k < n$, can also be viewed as a refinement of
a split $\Delta_{k+1,n}$.

It was shown by Worsey and Farin~\cite{worsey1987ann} that
$S^1_2(\Delta_{2,n})$ is degenerate for any $n\geq 2$.
Based on this observation, they concluded, as `an interesting aside', that
their $C^1$ piecewise-cubic element has supersmoothness of
order $2$ at $v$.
Theorem~\ref{thm:charac} implies more generally that
$S^1(\Delta_{2,n}) \subset C^2(\Delta_{2,n})$.
Using now degeneracy over the
Alfeld split in $\RR^k$ we obtain a more general result.

\begin{theorem}\label{thm:mos-WF}
The maximal order of supersmoothness of a $\Delta_{k,n}$ split,
$2 \le k \le n$, is bounded as follows:
$$ \rho_{k,r} \le \mos S^r(\Delta_{k,n}) \le \rho_{n,r}. $$
\end{theorem}

\begin{proof}
First let $r \le d \le \rho_{k,r}$.
We will show that $S_d^r(\Delta_{k,n})$ is degenerate.
The proof of this is similar to that of~\cite[Theorem 3.2]{worsey1987ann}
and is by induction on $n \ge k$.
Consider first $n=k$. Since $\Delta_{k,k}$ is a $k$-dimensional
Alfeld split it follows from
Lemma~\ref{lem:celldim} that $S_d^r(\Delta_{k,k})$ is degenerate.
Now suppose $n > k$ and let $s \in S_d^r(\Delta_{k,n})$.
Let $F$ be one of the $(n-1)$-faces of $T$.
Let $w$ be the point in the relative interior of~$F$
used to make the $(n-1)$-dimensional
split $\Delta_{k,n-1}(F)$ of $F$ in the construction of $\Delta_{k,n}$.
For each $\lambda \in (0,1]$, let
$F_\lambda$ be the $(n-1)$-simplex
$$ F_\lambda = \{(1-\lambda)v + \lambda x: x \in F \}, $$
which is parallel to $F$ and passes through the point
$$ p = (1-\lambda) v + \lambda w. $$
The split $\Delta_{k,n-1}(F)$ induces an analogous split
$\Delta_{k,n-1}(F_\lambda)$.
By the induction hypothesis,
$S_d^r(\Delta_{k,n-1}(F_\lambda))$ is degenerate and
so all the pieces of $s$ meeting at $[v,w]$
have common derivatives 
within $F_\lambda$ up to order $d$ at $p$.
Since all these pieces join continuously along
$[v,w]$, they also have common derivatives
along $[v,w]$.
Therefore all these pieces are the same polynomial and
thus $s$ belongs to $S_d^r(\Delta_{A,n})$.
Since $d \le \rho_{k,r} \le \rho_{n,r}$, it follows,
as in the proof of Theorem~\ref{thm:dim-Alfeld}, that $s \in \Pi_d$.

This proves the lower bound on $\mos S^r(\Delta_{\WF,n})$.
To prove the upper bound we just need to observe that
$\Delta_{k,n}$ is a refinement of an Alfeld split
$\Delta_{A,n}$, which implies that
$$ S_d^r(\Delta_{A,n}) \subset S_d^r(\Delta_{k,n}) $$
for any $0 \le r \le d$.
Thus if $S_d^r(\Delta_{A,n})$ is non-degenerate,
so is $S_d^r(\Delta_{k,n})$.
\end{proof}

%%%%%%%%%%%%%%%%%%%%%%%%%%%%%%%%%%%%%%%%%%%%%%%%%%%%%%%%%%
\subsection{The $\Delta_{n-1,n}$ split}
%%%%%%%%%%%%%%%%%%%%%%%%%%%%%%%%%%%%%%%%%%%%%%%%%%%%%%%%%%
Consider the special case of the $\Delta_{n-1,n}$ split,
which has $n(n+1)$ subsimplices.
It can be constructed by
first making an Alfeld split $\Delta_{A,n}$ ($=\Delta_{n,n}$)
of an $n$-simplex $T$ using some interior point $v$.
We then choose an interior point of each boundary face $F$
(an $(n-1)$-simplex) of $T$ and use it to split $F$ into
$n$ subsimplices and then connect them to~$v$.

Let us say that $\Delta_{n-1,n}$ is \emph{aligned}
if, for every face $F$, the splitting point chosen for $F$ is the unique
point in $F$ that is collinear with $v$ and the
vertex of $T$ opposite~$F$.
This is what
Schenck and Sorokina \cite{schenck2018subdivision}
called a facet split.

\begin{theorem}\label{thm:dim-facet}
The maximal order of supersmoothness of an aligned split
$\Delta_{n-1,n}$ is
$$ \mos S^r(\Delta_{n-1,n}) = \rho_{n-1,r}. $$
\end{theorem}

\begin{proof}
The dimensions of the polynomial spline spaces for
an aligned split $\Delta_{n-1,n}$
were derived by Schenck and Sorokina~\cite{schenck2018subdivision}.
For $0 \le r \le d$,
\begin{align}\label{dim:F}
\dim S_d^r(\Delta_{n-1,n}) = \dim \Pi_d + A(n, d, r) + (n+1)P(n, d, r),
\end{align}
where $A(n, d, r)$ is as in Theorem~\ref{thm:dim-Alfeld} and 
$$
P(n, d, r)=\begin{cases}
(n-1)\binom{d+n-(r+1)n/2}{n}, \quad \mbox{if $r$ is odd};\\
\binom{d+n-1-rn/2}{n}+\cdots +\binom{d+1-rn/2}{n},
  \quad \mbox{if $r$ is even}.
\end{cases}
$$
Therefore, $S_d^r(\Delta_{n-1,n})$ is degenerate if and only if
$A(n, d, r) + (n+1)P(n, d, r) = 0$. Since
$A(n, d, r) = 0$ when $P(n, d, r) = 0$, this is equivalent to
the condition that $P(n, d, r) = 0$, which holds when
$$
\begin{cases}
{d-(r+1)n/2\leq -1}, \quad \mbox{if $r$ is odd};\\
{d-1-rn/2 \leq -1}, \quad \mbox{if $r$ is even}.
\end{cases}$$
The largest possible $d$ in both cases gives the result
by Corollary \ref{cor:mos}.
\end{proof}

It is remarked in~\cite[Remark 4.3]{schenck2018subdivision} that for $r=1$,
the dimension formula \eqref{dim:F} also holds even without the
collinearity condition, from which we conclude that
for an arbitrary $\Delta_{n-1,n}$ split,
$$ \mos S^1(\Delta_{n-1,n}) = \rho_{n-1,1} = n-1. $$
For example, in 3D, for an arbitrary
Worsey-Farin split $\Delta_{2,3}$ we have
$$ \mos S^1(\Delta_{2,3}) = 2. $$

%%%%%%%%%%%%%%%%%%%%%%%%%%%%%%%%%%%%%%%%%%%%%%%%%%%%%%%%%%%%%
\subsection{$2$-cells}
%%%%%%%%%%%%%%%%%%%%%%%%%%%%%%%%%%%%%%%%%%%%%%%%%%%%%%%%%%%%%
Finally, we consider a slightly different kind of cell,
constructed as follows.
Let $T$ be an $n$-dimensional simplex and choose an interior point $v$
of $T$ and connect it to just one $(n-1)$-face of $T$, forming
a simplex $T_1$ contained in $T$.
We now let $T_2$ be the the closure of $T \setminus T_1$.
The two elements $T_1$ and $T_2$ form what we will call a 2-cell,
$\Delta_2 = \{T_1,T_2\}$.
Of course it is not a cell of simplices because $T_2$ is not
a simplex.
Figure~\ref{fig:2-cell} shows a 2-cell in $2D$.
\begin{figure}[ht]
\centering
\includegraphics[height=0.14\textheight]{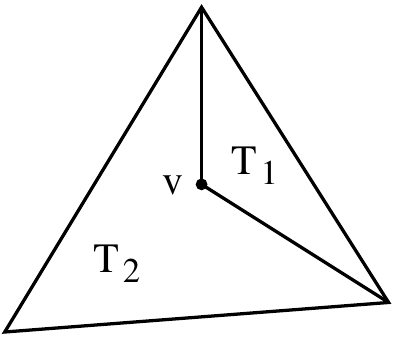}
\caption{A 2-cell in 2D.}
\label{fig:2-cell}
\end{figure}

Now we can consider the supersmoothness of splines
in $S^r(\Delta_2)$.
Even though 2-cells do not occur in
simplicial meshes,
the local configuration of edges emanating from $v$
could occur in a polytopal mesh if we allowed non-convex polytopes.
Shekhtman and Sorokina \cite{shekhtman2015intrinsic} studied
this kind of configuration in 2D and showed
that the order of supersmoothness is at least $r+1$
for any $r \ge 0$
(supersmoothness is `true' supersmoothness in this case,
not just the matching of derivatives).
We can now extend this result using our characterization.
Even though a 2-cell contains the non-simplicial element $T_2$,
the intersection of $T_1$ and $T_2$ is the union
of $n$ faces (of dimension $(n-1)$) and so our characterization
of supersmoothness at $v$ also holds for a $2$-cell,
i.e., we can apply Theorem~\ref{thm:charac}
and Corollary~\ref{cor:mos} to a $2$-cell $\Delta_2$.
To use these results we need the dimensions of
the spline spaces $S_d^r(\Delta_2)$, $0 \le r \le d$.

\begin{lemma}\label{lem:2-cell}
For any $0 \le r \le d$,
$$ \dim S_d^r(\Delta_2) = \dim \Pi_d + \dim \Pi_{d - n(r+1)}. $$
\end{lemma}

\begin{proof}
We have
$$ \dim S_d^r(\Delta_2) = \dim \Pi_d + \dim S_0, $$
where
$$ S_0 = \{ s \in S_d^r(\Delta_2): \hbox{$s \equiv 0$ on $T_2$} \}. $$
Letting $F_1,\ldots,F_n$ be the $(n-1)$-dimensional faces
common to $T_1$ and $T_2$, we have
$$ S_0 = \{ p \in \Pi_d: D^\alpha p|_{F_i} = 0, 
                 \, |\alpha| \le r, \, i=1,\ldots, n \}. $$
Let $l_i(x)=0$ be any equation for the face $F_i$, $i=1,\ldots, n$.
Then we can express any $p \in S_0$ uniquely in the form
$$ p(x) = l_1(x)^{r+1} \cdots l_n(x)^{r+1} q(x), \quad x \in T_1, $$
where $q = 0$ if $d-n(r+1) < 0$ and
$q \in \Pi_{d-n(r+1)}$ if $d-n(r+1) \ge 0$.

\end{proof}

\begin{theorem}\label{thm:2-cell}
For $r \ge 0$, $\mos S^r(\Delta_2) = r + (n-1)(r+1)$.
\end{theorem}

\begin{proof}
By Lemma~\ref{lem:2-cell},
$S_d^r(\Delta_2)$ is degenerate if and only if
$\dim \Pi_{d - n(r+1)} =0$, or equivalently
that $d - n(r+1) < 0$.
Thus, from Corollary~\ref{cor:mos},
the maximal order of supersmoothness is
$$ d = n(r+1)-1 = r + (n-1)(r+1). $$
\end{proof}

For example, in $\RR^2$, $\mos S^r(\Delta_2) = 2r+1$
and in $\RR^3$, $\mos S^r(\Delta_2) = 3r+2$.

%%%%%%%%%%%%%%%%%%%%%%%%%%%%%%%%%%%%%%%%%%%%%%%%%%%%%%%%%%%%%
%\section{Conclusions}\label{sec:conclusions}
%%%%%%%%%%%%%%%%%%%%%%%%%%%%%%%%%%%%%%%%%%%%%%%%%%%%%%%%%%%%%
 
%We have shown that the issue of supersmoothness is closely related
%to the dimension of polynomial spline spaces.
%Conversely, we also hope that the connection between these two aspects established in this paper may provide a different point of view for establishing dimension bounds of polynomial spline spaces. 
%
%Throughout this paper, we assume that the spline functions are piecewise $C^{\infty}$ for ease of representation. In particular, this is true for polynomials splines. It is easy to see that $C^{\rho}$ continuity on each piece is enough to let us conclude with the $\tilde{C}^{\rho}$ supersmoothness. 

%The discussions in this paper may shed some light on the construction of finite elements and spline spaces for scalar and vector valued functions. For example, we have shown that the $C^{2}$ continuity at the interior point of the Worsey-Farin split is intrinsic, meaning that we can treat the Worsey-Farin macroelement as a uniform $C^{1}$ spline. This will render convenience and flexibility when we construct vector fields  which are compatible with the Worsey-Farin $C^{1}$ macroelement. 

%%%%%%%%%%%%%%%%%%%%%%%%%%%%%%%%%%%%%%%%%%%%%%%%%%%%%%%%%%%%%%%%
%\bibliographystyle{plain}      % mathematics and physical sciences
\bibliography{intrinsic}{}   % name your BibTeX data base
%%%%%%%%%%%%%%%%%%%%%%%%%%%%%%%%%%%%%%%%%%%%%%%%%%%%%%%%%%%%%%%%

\end{document}